\documentclass[10pt]{amsart}
\usepackage{amssymb}
\usepackage{amsmath}
\usepackage{amsthm}
\usepackage{latexsym}
\usepackage{graphicx}
\usepackage{hyperref}
\newcommand{\Z}{{\mathbb Z}}
\newcommand{\R}{{\mathbb R}}
\newcommand{\C}{{\mathbb C}}

\newcommand{\T}{{\mathbb T}}

\newtheorem{lemma}{Lemma}[section]
\newtheorem{theorem}[lemma]{Theorem}

\newtheorem{remark}[lemma]{Remark}
\newtheorem{proposition}[lemma]{Proposition}

\newtheorem{definition}[lemma]{Definition}

\newcommand{\be}{\begin{equation}}
\newcommand{\ee}{\end{equation}}

\newcommand{\ti}{\tilde}

\newcommand{\spr}[2]{\left\langle #1 , #2 \right\rangle}

\newcommand{\E}{\mathrm{e}}
\newcommand{\I}{\mathrm{i}}

\newcommand{\tr}{\mathrm{tr}}
\newcommand{\im}{\mathrm{Im}}

\newcommand{\eps}{\varepsilon}

\numberwithin{equation}{section}

\begin{document}

\title[Limit-periodic Schr\"odinger operators]{Absolutely continuous spectrum for limit-periodic Schr\"odinger operators}

\author[H.\ Kr\"uger]{Helge Kr\"uger}
\address{Mathematics 253-37, Caltech, Pasadena, CA 91125}
%\address{Erwin Schr\"odinger Institute, Boltzmanngasse 9, A-1090 Vienna, Austria}
% \address{Department of Mathematics, Rice University, Houston, TX~77005, USA}
\email{\href{helge@caltech.edu}{helge@caltech.edu}}
\urladdr{\href{http://www.its.caltech.edu/~helge/}{http://www.its.caltech.edu/~helge/}}

\thanks{H.\ K.\ was supported by the Simons Foundation.}

\date{\today}

%\subjclass[2000]{Primary 81Q10; Secondary 37D25}

\begin{abstract}
 We show that a large class of limit-periodic Schr\"odinger operators has
 purely absolutely continuous spectrum in arbitrary dimensions. This result was previously
 known only in dimension one.

 The proof proceeds through the non-perturbative construction of limit-periodic
 extended states. An essential step is a new estimate of the probability (in
 quasi-momentum) that the Floquet Bloch operators have only simple
 eigenvalues.
\end{abstract}

\maketitle

%%%%%%%%%%%%%%%%%%%%%%%%%%%%%%%%%%%%%%%%%%%%%%%%%%%%%%%%%%%%%%%%%%%%%
%
%
%

\section{Introduction}

In this paper, we consider Schr\"odinger operators $\Delta + V$
acting on the lattice $\ell^2(\Z^d)$ for $d\geq 1$. Here $\Delta$
is the discrete Laplacian
\be
 \Delta\psi(n)=\sum_{|e|_1=1} \psi(n+e),\quad |x|_1=|x_1|+\dots+|x_d|
\ee
and the potential $V$ is a multiplication operator by a sequence
$V:\Z^d\to\R$. For some general background, see Sections 3 and 4
in \cite{kirsch}. The potential
$V$ is called $p=(p_1, \dots, p_d)$-periodic if
\be
 V(n_1 + p_1, \dots, n_d) = \dots = V(n_1, \dots, n_d + p_d)
 = V(n_1, \dots, n_d)
\ee
for all $n\in\Z^d$. A sequence of periods $p^1, p^2, \dots$
is called increasing if $p^{\ell}_j$ divides $p^{\ell + 1}_j$
for all $\ell\geq 1$ and $j =1,\dots, d$. $V$ is limit-periodic
if there exists an increasing sequence of periods $p^j$ and
$p^j$-periodic potentials $V^j$ such that
\be
 V_j = V^1 + \dots V^j
\ee
converges to $V$ in $\ell^{\infty}(\Z^d)$.

 The main result is

\begin{theorem}\label{thm:ac}
 Let $d\geq 1$, $\eps_1 > 0$, and $p^j$ be an increasing sequence of periods.
 Then there exists a sequence $\eps_j > 0$, $j\geq 2$, such that for
 $V^j$ a $p^j$-periodic potential satisfying 
 $\|V^j\|_{\ell^{\infty}(\Z^d)}\leq\eps_j$, the potential
 \be
  V = \lim_{j\to\infty} (V^1 + \dots V^j)
 \ee
 exists in $\ell^{\infty}(\Z^d)$ and the Schr\"odinger
 operator $\Delta + V$ has purely absolutely continuous
 spectrum.
\end{theorem}

This statement was originally proven by Avron and Simon \cite{as1}
for Schr\"odinger operators on $L^2(\R)$. Damanik and Gan \cite{dg1}
gave a proof for the case of $\ell^2(\Z)$. As far as higher
dimensional Schr\"odinger operators are concerned, Karpeshina
and Lee \cite{kl1} have shown the existence of an
absolutely continuous component of the spectrum in the perturbative regime
of high energies on $L^2(\R^2)$. So the results are new for
$d\geq 2$. Furthermore, the proof given here is much simpler
than the strategy of Karpeshina and Lee.

In difference to Karpeshina and Lee, we do not discuss the
spectrum of $H$ as a set. The main reason is that our results allow
for the spectrum to contain many gaps, just start with a large
enough $V_1$. Finally, limit-periodic Schr\"odinger operators
with pure-point spectrum have been constructed by Damanik and
Gan in \cite{dg2} in arbitrary dimension. Finally, the results
of \cite{b2007} and Chapter~17 in \cite{bbook} imply the
existence of extended states for quasi-periodic Schr\"odinger
operators in arbitrary dimension and small coupling for
a set of frequencies of large measure.

The proof proceeds by constructing generalized eigenfunctions,
that is solutions $u:\Z^d\to\C$ of $H u = E u$. We will show

\begin{theorem}\label{thm:genev}
 Let $V$ be as in Theorem~\ref{thm:ac}. Then
 for almost every $\theta_1,\dots,\theta_d\in\R$, there exists
 $E\in\R$ and non-zero limit-periodic $u:\Z^d\to\C$ such that
 \be
  H u = Eu
 \ee
 and
 \be
  \hat{u}(\theta_1, \dots, \theta_n) = \lim_{R\to\infty}
   \frac{1}{\#\Lambda_R(0)} \sum_{n\in\Lambda_R(0)}
    u(n) e(n_1 \theta_1 + \dots + n_d \theta_d)
     \neq 0.
 \ee
\end{theorem}

Here, we use the notation $e(x) = \E^{2\pi\I x}$
and $\Lambda_R(n)=\{x\in\Z^d:\quad |n-x|_{\infty}\leq R\}$.
We will now discuss properties single periodic operator
following \cite{kper}.
Given a period $p \in (\Z_+)^d$, we introduce the set
\be
 \mathbb{B}_{p} = \left\{(\frac{k_1}{p_1},\ \dots,\ \frac{k_d}{p_d}),
  \quad {0\leq k_j\leq p_j-1} \right\}.
\ee
Any $p$-periodic function $V$ can be written as
\be
 V(n) = \sum_{k\in\mathbb{B}_{p}} \widehat{V}(k) 
  e(k \cdot n),
\ee
where $x\cdot y = \sum_{j=1}^{d} x_j y_j$. For $u\in\ell^1(\Z^d)$,
we define the Fourier transform $\hat{u}:\T^d\to\R$, $\T=\R/\Z$ by
\be
 \hat{u}(x) = \sum_{n\in\Z^d} u(n) e(x\cdot n).
\ee
This map is extended to $\ell^2(\Z^d)\to\ell^{2}(\Z^d)$ as usual.
Furthermore, the Fourier transform of $(\Delta + V) u$ is
given by
\be
 \sum_{j=1}^{d} 2\cos(2\pi x_j) \hat{u}(x)
  + \sum_{k\in\mathbb{B}_{p}} \widehat{V}(k) \hat{u}(x + k).
\ee
Letting $\psi_x = \{\hat{u}(x+k)\}_{k\in\mathbb{B}_{p}}$,
we see that the action of this operator is equivalent to
\be
 \widehat{H}_{x} \psi(k) = \sum_{j=1}^{d} 2\cos(2\pi (x_j+k_j)) \psi(k)
  + \sum_{\ell\in\mathbb{B}_{p}} \widehat{V}(\ell) \psi(k+\ell). 
\ee
The operator $\widehat{H}_{x}$ acts on the $P = p_1\cdots p_d$
dimensional space $\ell^2(\mathbb{B}_{p})$, and we can uniquely define its eigenvalues
by
\be
 E(x,1) \leq E(x,2) \leq \dots \leq E(x,P).
\ee

\begin{definition}
 Let $\delta > 0$. The spectrum of $\widehat{H}_{x}$ is called
 $\delta$-simple if for every $1\leq\ell\leq P-1$, we have
 \be
  E(x,\ell+1)-E(x,\ell)\geq\delta.
 \ee
 The spectrum of $\widehat{H}_x$ is called simple if it is $\delta$-simple
 for some $\delta > 0$.
\end{definition}

For the $x$ such that the spectrum of $\widehat{H}_x$ is simple, we can choose
normalized eigenfunctions $\psi(x,\ell)$ of $\widehat{H}_x$ such that
\be
 \widehat{H}_x \psi(x,\ell) = E(x,\ell) \psi(x,\ell).
\ee
Finally, the map $(x,\ell)\mapsto E(x,\ell)$ is continuous
and the map $(x,\ell)\mapsto \psi(x,\ell)$ can be chosen to
be continuous at least on the set of simple spectrum.
The main technical ingredient in our proofs will be the
following theorem.

\begin{theorem}\label{thm:simple}
 Let $V$ be a $p$-periodic potential. Given $\eta\in(0,\frac{1}{2})$, there
 exists a set $\mathcal{G}\subseteq\T^d$ and $\delta = \delta(\eta, \|V\|_{\infty}, p) >0$ such that 
 \begin{enumerate}
  \item $|\mathcal{G}|\geq 1 - \eta$.
  \item For $x\in\mathcal{G}$, we have that the spectrum of
   $\widehat{H}_{x}$ is $\delta$-simple.
 \end{enumerate}
\end{theorem}

A more detailed statement is given in Section~\ref{sec:simplespec}.
In particular, the dependance of $\delta$ on $\eta$ is quantitative and
given by $\delta = \eta^{C \log(P) P^2}$ for a constant $C > 1$.
In fact, the contents of that section are the main technical steps
in the proof of Theorems~\ref{thm:ac} and \ref{thm:genev}.
Before deducing how to prove Theorem~\ref{thm:genev}, we give a
non-quantitative argument that implies Theorem~\ref{thm:simple}
for some $\delta > 0$.

Define the discriminant 
\be\label{eq:deff}
 f(x) = \prod_{j<\ell} (E_j(x)-E_{\ell}(x))^2.
\ee
For $x\in\R^d$, we have that $|E_j(x)|\leq\|\widehat{H}_x\|
\leq 2d+\|V\|_{\infty}$. Thus, we obtain that
\be
 \min_{j\neq \ell} |E_{j}(x) - E_{\ell}(x)| \geq
  \frac{|f(x)|^{\frac{1}{2}}}{(2d+\|V\|_{\infty})^{\frac{P^2}{2}}}.
\ee
Furthermore we have that $f(x)=(-1)^{\frac{1}{2}P(P-1)}\mathrm{Res}(
P(\cdot,x), \partial_{E}P(\cdot,x))$ for $P(E,x) = \det(E-\widehat{H}_x)$,
where $\mathrm{Res}$ denotes the resultant. As the resultant is
a polynomial in the coefficients $c_j(x)$ of $P(E,x)=E^{P}+\sum_{j=0}^{P-1} c_j(x) E^j$,
it follows that $f(x)$ is analytic. The following is a qualitative
implementation of Theorem~\ref{thm:caresti}.

\begin{proof}[Proof of Theorem~\ref{thm:simple}]
 By Proposition~\ref{prop:evs} (iii), we have that $f(z)\neq 0$ for some $z\in\C^d$.
 This implies that the map $g_1: x_1 \mapsto f(x_1, z_2, \dots, z_d)$ is
 analytic and not equal to zero, thus $|g_1(x_1)| \geq \kappa_1$ 
 for all $x_1\in [0,1]\setminus X_1$ with $|X_1| \leq\eta/d$ for 
 some $\kappa_1 > 0$. Applying this construction to $g_j: x_j \mapsto
 f(x_1, \dots, x_{j}, z_{j+1}, \dots, z_{d})$ for
 $x_{\ell}\in [0,1] \setminus X_{\ell}$, we obtain a sequence of
 sets $X_j$ with $|X_j|\leq\eta/d$ and $|g_j(x_j)|\geq\kappa_j >0$
 for $x_j\notin X_j$. Taking
 \[
  \mathcal{G} = ([0,1]\setminus X_1)\times \dots \times ([0,1]\setminus X_d)
 \]
 the claim follows.
\end{proof}

We now start with the proof of Theorem~\ref{thm:genev}.
Denote by $\widehat{H}^{j}_{x}$ the $p^j$-periodic operator
with potential $V^j = V_1 + \dots + V_j$. Let us assume for
a second that $V_{j+1} = 0$ and try to understand the relation
of $E^{j}(x,\ell)$ and $E^{j+1}(x,\ell)$. As sets, we clearly
have that
\be
 \sigma(\widehat{H}^{j+1}_x) = \bigcup_{s\in\mathbb{S}_{j+1}} 
  \sigma(\widehat{H}^{j}_{x+s}),
   \quad\sigma(\widehat{H}^{j}_x)=\{E^{j}(x,\ell)\}_{\ell=1}^{P_j}
\ee
where 
\be
 \mathbb{S}_{j+1} = \left\{\left(\frac{s_1}{p^{j+1}_1},\dots,
\frac{s_d}{p^{j+1}_d}\right),\quad 0\leq s_k \leq 
\frac{p^{j+1}_k}{p^{j}_k}-1\right\}.
\ee
If the spectrum of $\widehat{H}_{x}^{j+1}$ is simple, we 
thus clearly have that
there exists for each $1\leq \ell\leq P_{j+1}$ an unique
$1\leq\ti\ell\leq P_{j}$ and $s\in\mathbb{S}_{j+1}$ such that
\be
 E^{j+1}(x,\ell) = E^{j}(x + s, \ti\ell)
\ee
and $\psi^{j+1}(x,\ell) = c \psi^{j}(x+s,\ti\ell)$ for some
$|c|=1$. 

\begin{remark}
 In order to understand, the equality $\psi^{j+1}(x,\ell) = 
 c \psi^{j}(x+s,\ti\ell)$, we view $\psi^{j}(x,\ell)$ as an
 element of $\ell^2(\mathbb{B}_{p^{j}}+x)$. Then
 as $\mathbb{B}_{p^{j}} + x + s \subseteq
 \mathbb{B}_{p^{j+1}} + x$ for $s\in\mathbb{S}_{j+1}$ the equality
  makes sense in $\ell^2(\mathbb{B}_{p^{j+1}}+x)$.
 These are natural choices given the definition
 of $\widehat{H}_x^j$. Finally, we have that
 \be
  \mathbb{B}_{p^{j+1}}=\bigcup_{s\in\mathbb{S}_{j+1}}(\mathbb{B}_{p^{j}} + s)
 \ee 
 and $(\mathbb{B}_{p^{j}} + s) \cap (\mathbb{B}_{p^{j}} + \ti{s})=\emptyset$
 for $s,\ti{s}\in\mathbb{S}_{j+1}$ and $s\neq\ti s$.
\end{remark}

Let us now consider the case of $V_{j+1} \neq 0$. For this,
we will assume that the spectrum of $\widehat{H}_{x}^{j+1}$
is $\delta$-simple for some $\delta > 0$. Then if $\|V_{j+1}\|_{\infty} \leq
\frac{\delta}{3}$, we got for the same identification
$\ell \mapsto (s,\ti\ell)$ that
\be
 |E^{j+1}(x,\ell)-E^{j}(x+s,\ti\ell)| \leq \|V_{j+1}\|_{\infty}.
\ee
Thus we have by Theorem~\ref{thm:distevs}
\be
 d(\psi^{j+1}(x, \ell), \psi^{j}(x+s,\ti\ell))
  \leq \frac{2}{\delta} \|V_{j+1}\|_{\infty}
\ee
where 
\be
 d(\psi,\varphi) = \inf_{|c|=1} \|\psi - c \varphi\|
\ee
is the distance between normalized eigenfunctions.
We define the parametrizing set
\be
 \mathbb{P}_{j} = \mathbb{V}_j \times \{1,\dots,P_j\},\quad
  \mathbb{V}_{j} = [1, \frac{1}{p^j_1}) \times \dots \times
   [1, \frac{1}{p^{j}_{d}}).
\ee
Clearly $|\mathbb{P}_{j}| = 1$.
In order to state our main result, we introduce
$\eta_j = 2^{-j}$, $\delta_j$ is the $\delta$ obtained
from Theorem~\ref{thm:simple}, and $\eps_j = (\delta_j)^{10}$.

\begin{theorem}
 Assume $\|V_{j+1}\|_{\infty} \leq \eps_{j+1}$. Then there exists
 $\mathbb{G}_{j+1} \subseteq \mathbb{P}_{j+1}$ and
 a map $A_{j}: \mathbb{G}_{j+1} \to \mathbb{P}_{j}$ such that
 \begin{enumerate}
  \item $|\mathbb{G}_{j+1}| \geq 1-\eta_j$.
  \item For $(x,\ell) \in\mathbb{G}_{j+1}$, we have that
   $A_{j}(x,\ell) = (x + s, \ti{\ell})$ for some $s\in\mathbb{S}_{j+1}$,
   $\ti\ell\in\{1,\dots,P_{j}\}$.
  \item The map $A_{j}$ is continuous.
  \item For $(x,\ell)\in\mathbb{G}_{j+1}$, we have
   \be
    |E^{j+1}(x,\ell) - E^{j}(A_j(x,\ell))| \leq \eps_{j+1}
   \ee
   and
   \be
    d(\psi^{j+1}(x,\ell),\psi^{j}(A_j(x,\ell)))
     \leq \frac{2 \eps_{j+1}}{\delta_{j+1}}.   
   \ee
 \end{enumerate}
\end{theorem}

\begin{proof}
 This is essentially, what we have discussed above.
\end{proof}

We have seen that if for $(x,\ell)\in\mathbb{P}_{j}$,
there exists $(\ti{x},\ti{\ell})\in\mathbb{G}_{j+1}$
such that $(x,\ell) = A_j (\ti{x},\ti{\ell})$ then this
$(\ti{x}, \ti{\ell})$ is unique. Finally, we have that
$|A_j(\mathbb{G}_{j+1})| = |\mathbb{G}_{j+1}|$. Hence, for
any $j$, we have that the set
\be
 \mathcal{G}_{j} = \bigcap_{k \geq j}
  A_j \cdots A_k \mathbb{G}_{k+1}
\ee
has measure
\be
 |\mathcal{G}_{j}| \geq 1 - \sum_{k\geq j}^{\infty} \eta_k \geq 1 - 2\eta_j.
\ee
We also note that $\mathcal{G}_{j}\subseteq\mathcal{G}_{j+1}$.
For $(x,\ell) \in \mathcal{G}_j$, we obtain a sequence
$(x_k, \ell_k)$ such that
\be
 (x,\ell) = A_j \cdots A_k (x_k, \ell_k)
\ee
and we have that the eigenfunctions and eigenvalues converges.
In particular that
\be
 d(\psi^{k}(x_{k},\ell_{k}), \psi^{\ti{k}}(x_{\ti{k}},\ell_{\ti{k}})) \leq 2 \delta_k^{9}
\ee
for $\ti{k}\geq k\geq j$.

\begin{proof}[Proof of Theorem~\ref{thm:genev}]
 As the convergence is fast enough to also imply convergence
 in the $\ell^1$ norm, i.e. for the sequence $(x_{\ell}, \ell_k)$
 corresponding to $(x,\ell)\in\mathcal{G}_j$, we have
 \[
  \|\psi^{k}(x_k,\ell_k) - \psi^{j}(x,\ell)\|_{\ell^1} \leq 2 \delta_k^{9}.
 \]
 Define
 \[
  \varphi^{k}(n) = \sum_{t\in\mathbb{B}}\psi^{k}(x_k,\ell_k; t)
   e(-t\cdot n).
 \]
 Then we have that the $\varphi^k$ converge to a limit $\varphi$ in
 $\ell^{\infty}(\Z^d)$ and $(H^{k} - E(x_k, \ell_k)) \varphi^k = 0$.
 Letting $E = \lim_{k\to \infty} E(x_k, \ell_k)$, we find
 \[
  (H - E) \varphi = 0.
 \]
 Finally, by construction it is easy to see that we can satisfy
 the frequency condition for all $x$ such that $(x,\ell)\in\mathcal{G}_j$
 for some $\ell$. As $|\mathcal{G}_j| \to 1$, the claim follows.
\end{proof}

In order to prove Theorem~\ref{thm:ac}, we will need a sharpening
of Theorem~\ref{thm:simple}, which we present in the
following section. Then, we proceed to prove Theorem~\ref{thm:ac}.

%%%%%%%%%%%%%%%%%%%%%%%%%%%%%%%%%%%%%%%%%%%%%%%%%%%%%%%%%%%%%%%%%%%%%%%%%%%%%%%%%%%
%
%
%

\section{Simple spectrum}\label{sec:simplespec}

The goal of this section is to prove a sharpening
of Theorem~\ref{thm:simple}.

\begin{theorem}\label{thm:simple2}
 Let $V$ be a $p$-periodic potential. Given $\eta \in (0,\frac{1}{2})$, there
 exists a set $\mathcal{G}\subseteq\T^d$ such that 
 \begin{enumerate}
  \item $|\T^d\setminus\mathcal{G}|\leq \eta$.
  \item For $x\in\mathcal{G}$, we have the spectrum of
   $\widehat{H}_x$ is $\delta$-simple for
   \be
    \delta = \left(\eta\right)^{C P^2\log(P)}
   \ee
   for some $C > 1$ that only depends on $d$ and $\|V\|_{\infty}$.
  \item For $x\in\mathcal{G}$, we have that 
   $|\partial_{x_d} E(x,\ell)| \geq \gamma$ for
   \be
    \gamma = \left(\eta\right)^{C P^2\log(P)}.
   \ee
   for some $C >  1$ that only depends on $d$ and $\|V\|_{\infty}$.
 \end{enumerate}
\end{theorem}

In order to prove this theorem, we will need to gain further
understanding of the operator $\widehat{H}_x$. We begin by proving
a simple proposition, which we will need for the study
of the absolutely continuous spectrum and whose proof
introduces some techniques necessary to prove Theorem~\ref{thm:simple2}.

\begin{proposition}\label{prop:boundxd}
 Let $V$ be a $p$-periodic potential, $x' \in
 [0, (p_1)^{-1})\times \dots [0,(p_{d-1})^{-1})$, and
 $E\in\R$. Then
 \be
  \#\{x_d\in [0,(p_{d-1})^{-1}): E\in\sigma(\widehat{H}_{(x',x_d)})\}
   \leq 2 p_1 \cdots p_{d-1}.
 \ee
\end{proposition}

We define
\be
 P(x; E) = \det(\widehat{H}_x - E)
\ee
and observe that it is a trigonometric polynomial of degree
$P=p_1\cdots p_{d}$ in each of the $x_j$. Furthermore, we have
that $P(\ti{x}; E) = P(x; E)$ if $\ti{x}_j - x_j \in \frac{1}{p_j} \Z$.

\begin{proof}[Proof of Proposition~\ref{prop:boundxd}]
 $E\in\sigma(\widehat{H}_{(x', x_d)})$ is equivalent
 to $g(x_d) = P(x', x_d; E) = 0$. Now as $g$ is a trigonometric
 polynomial of degree $P$, we have that
 \[
  \#\{x_d\in [0,1):\quad g(x_d) = 0\} \leq 2 P.
 \]
 As the number $\#\{x_d\in [t p_{d}^{-1}, (t+1) p_{d}^{-1})\quad g(x_d) = 0\}$
 is constant in $t$, the claim follows.
\end{proof}

We will need to consider $x$ not just in $[0,1]^d$ but in
the entire complex plane $\C^d$. We will denote in this
section by $E_j(x)$ the eigenvalues of $\widehat{H}_x$.
We collect the properties of these eigenvalues in.

\begin{proposition}\label{prop:evs}
 \begin{enumerate}
%  \item For $x\in\R^d$, we have $|E_j(x)|\leq 2d+\|V\|_{\infty}$.
  \item For $z\in \C^d$, we have that
   $|E_j(z)|\leq d + \sum_{j=1}^{d} \E^{|\im(z_j)|} + \|V\|_{\infty}$.
  \item For $z\in\C^d$, we have $|\partial_{z_d} E_j(z)|
   \leq 2\pi ( \E^{|\im(z_d)|} + 1)$.
  \item Let $y_j = \frac{1}{2\pi} \log(p_1\cdots p_j 2^j ( 4(d+\|V\|_{\infty}) + 1))$,
   and $z_j = \I y_j$. Then $|E_j(z) - E_{\ell}(z)|\geq 1$ for $j\neq\ell$.
  \item For $z_j$ as in (iv), we have that
   $|\partial_{z_d} E_j(z)| \geq \frac{1}{2} \E^{2\pi y_d}$.
 \end{enumerate}
\end{proposition}

We recall that $\widehat{H}_{x}: \ell^2(\mathbb{B}_{p})
\to\ell^2(\mathbb{B}_{p})$ is given by
$\widehat{H}_{x} = \widehat{H}^{0}_{x} + \widehat{V}$
with $\|\widehat{V}\|\leq \|V\|_{\infty}$ and
\be
 \widehat{H}_{x}^{0} \psi(k) = \sum_{j=1}^{d} 2\cos(2\pi (x_j+k_j)) \psi(k)
\ee
is a multiplication operator.

\begin{proof}[Proof of Proposition~\ref{prop:evs} (i), (ii)]
 This follows from the bound 
 \be
  \|\widehat{H}_{x}^{0}\| \leq \sum_{j=1}^{d}
   \cosh(\im(z_j)) \leq \sum_{j=1}^{d} \E^{|\im(z_j)|} + d.
 \ee
 As
 \be
  \partial_{x_d} \widehat{H}_{x}^{0} \psi(k) = - 4 \pi \sin(2\pi (x_d +k_d)) \psi(k)
 \ee
 and $\partial_{z_d} E_{j}(z) = \spr{\psi_j(z)}{\partial_{z_d} \widehat{H}_{z} \psi_j(z)}$
 also (ii) follows.
\end{proof}

We can write 
\be
 \widehat{H}_{x} = A(x) + B(x)
\ee
with $\|B(x)\| \leq d + \|V\|_{\infty}$ and $A(x)$ being
the diagonal matrix with entries
\be
 d(k, y) = \sum_{j=1}^{d} e(\frac{k_j}{p_{j}}) \E^{2\pi y_j},\quad
  k \in \{0, \dots, p_{1}-1\} \times \dots \times \{0,\dots, p_{d} - 1\}.
\ee

\begin{lemma}
 Let $A > 0$. Then for
 \be
  \E^{2\pi y_1} \geq \frac{A p_1}{2\pi},\quad
   \E^{2\pi y_{j}} \geq p_{j} \left(\frac{1}{\pi} + \frac{1}{p_{j-1}}\right)
    \E^{2\pi y_{j-1}}
 \ee
 we have that for $k \neq \ell$
 \be
  |d(k, y) - d(\ell, y)| \geq A.
 \ee
\end{lemma}

\begin{proof}
 Let $1\leq j\leq d$ be the largest choice such that
 $k_j \neq \ell_j$. Thus
 \[
  e(\frac{k_{j+1}}{p_{j+1}}) \E^{2\pi y_{j+1}} + \dots + 
  e(\frac{k_{d}}{p_{d}}) \E^{2\pi y_{d}} = 
  e(\frac{\ell_{j+1}}{p_{j+1}}) \E^{2\pi y_{j+1}} + \dots + 
  e(\frac{\ell_{d}}{p_{d}}) \E^{2\pi y_{d}}
 \]
 and
 \[
  |e(\frac{k_j}{p_j}) \E^{2\pi y_j} - 
  e(\frac{\ell_j}{p_j}) \E^{2\pi y_j}| \geq \frac{2\pi}{p_j} \E^{2\pi y_j}.
 \] 
 Thus, we are done if we choose $y_j$ such that
 \[
  \frac{2\pi}{p_j} \E^{2\pi y_j}\geq 2 ( \E^{2\pi y_1} + \dots + \E^{2\pi y_{j-1}}) + A
 \]
 holds.
\end{proof}

We see that with our choice of $y_j$, these bounds hold
with $A = d + \|V\|_{\infty} + 1$.

\begin{proof}[Proof of Proposition~\ref{prop:evs} (iii)]
 We have that the eigenvalues of $A(y)$ are at least
 \[
  d + \|V\|_{\infty} + 1
 \]
 apart. Hence, the claim follows by standard bounds.
\end{proof}

\begin{proof}[Proof of Proposition~\ref{prop:evs} (iv)]
 Let $E_j(y)$ be an eigenvalue of $\widehat{H}_{y}$. Then by
 the previous considerations. There exists an unique $k$ such that
 \[
  |E_j(y) - d(k,y)|\leq d + \|V\|_{\infty}.
 \]
 Hence for $\psi$ a normalized solution of $(\widehat{H}_{y} - E_j(y))
 \psi = 0$, we have that
 \[
  \|(A(y) - E_j(y)) \psi\| = \|(A(y) + B(y) - E_j(y)) \psi + B(y) \psi\|
   \leq d + \|V\|_{\infty}.
 \] 
 Now as
 \[
  \|(A(y) - E_j(y)) \psi\|\geq \sum_{\ell\neq k} |d(\ell, y) - E_j(y)| |\psi(\ell)|^2
   \geq \left(3 (d + \|V\|_{\infty}) + 1\right)\sum_{\ell\neq k} |\psi(\ell)|^2,
 \]
 we conclude that $|\psi(k)|^2 \geq \frac{2}{3}$ and
 the claim follows.
\end{proof}

We have already defined $f(z)$ in \eqref{eq:deff}.
We also define
\be
 g(z)= \mathrm{Res}(P(z; .), \partial_{z_d} P(z; .)),
\ee
which is also analytic and satisfies
\be\label{eq:gasdzdP}
 g(z) = \prod_{\ell} \partial_{z_d} P(z; E_\ell(z)).
\ee

\begin{lemma}
 We have that
 \be
  g(z) = f(z) \cdot \prod_{\ell} \partial_{z_d} E_{\ell}(z).
 \ee
\end{lemma}

\begin{proof}
 As $P(z, E_{\ell}(z)) = 0$, we have that
 \[
  \partial_{z_d} P(z; E_{\ell}(z)) = \partial_{z_d}
   E_{\ell}(z) \cdot \partial_{E} P(z; E_{\ell}(z)).
 \]
 Similarly to \eqref{eq:gasdzdP}, we have that
 $f(z)=\prod_{\ell} \partial_{E} P(z;E_{\ell}(z))$,
 so the claim follows.
\end{proof}

In the following, we will use the norm 
\be
 |z| = \max(|z_1|, \dots, |z_d|)
\ee
on $z\in\C^d$.

\begin{proposition}\label{prop:fg}
 \begin{enumerate}
  \item $|f(z)| \leq (4 d \E^{2\pi |z|} + \|V\|_{\infty})^{P^2}$.
  \item $|g(z)| \leq (4 \pi \E^{2\pi |z|})^{P} \cdot |f(z)|$.
  \item There exists $y$ with $1 \leq |y| \leq 
   \frac{1}{2\pi}\log(P 2^d ( 4d+4\|V\|_{\infty}+1))$ such that
   \be
    |f(y)| \geq 1, \quad |g(y)| \geq 1
   \ee
%  \item For $x\in\R^d$, we have
%   \be
%    |\partial_{E} P(x; E_{\ell}(x))| \geq \frac{|f(x)|}{(4d+2\|V\|_{\infty}+1)^{P(P-1)}}.
%   \ee
  \item For $x\in\R^d$, we have that 
   \be
   |\partial_{x_d} P(x; E_{\ell}(x))| \geq 
     \frac{|g(x)|}{(4d+2\|V\|_{\infty}+1)^{P(P-1)}}
   \ee
%  \item For $x\in\R^d$, we have
%   \be
%    \inf_{j\neq \ell} |E_{j}(x) - E_{\ell}(x)| \geq
%     \frac{|f(x)|}{(4d+2\|V\|_{\infty}+1)^{P^2}}.
%   \ee
 \end{enumerate}
\end{proposition}

\begin{proof}[Proof of Proposition~\ref{prop:fg} (i), (ii)]
 By Proposition~\ref{prop:evs} (i), we have that
 \[
  |E_j(z)|\leq d(1+\E^{|z|} +\|V\|_{\infty})
 \]
 (i) thus follows by \eqref{eq:deff}.
 (ii) now follows by the previous lemma.
\end{proof}

\begin{proof}[Proof of Proposition~\ref{prop:fg} (iii)]
 The lower bound on $f(y)$ follows by Proposition~\ref{prop:evs} (iv).
 In order to deduce the one on $g(z)$ use the previous lemma
 and Proposition~\ref{prop:evs} (v), and that $\frac{1}{2} \E^{2\pi y_d} \geq 1$.
\end{proof}

\begin{proof}[Proof of Proposition~\ref{prop:fg} (iv)]
 By \eqref{eq:gasdzdP}, we clearly have that
 \[
  |\partial_{x_d} P(x; E_{\ell}(x))| \geq |g(x)| \cdot
   \left(\max_{1\leq j\leq P}|\partial_{x_d} P(x;E_{j}(x))|\right)^{-(P-1)}.
 \]
 By Cauchy's integral formula, we obtain for $x=(x',x_d)$
 \[
  \partial_{x_{d}} P(x;E_{j}(x)) = -\frac{1}{2\pi\I} 
   \int_{|t-x_d|=1} \frac{P(x',t,E_{j}(x))}{(t-x_d)^2} dt.
 \]
 As $\|\widehat{H}(x) - E_j(x)\|\leq 4d+2\|V\|_{\infty} + 1$, the claim follows.
\end{proof}

Finally, we observe

\begin{lemma}
 For $|z| \leq 4 e |y|$, we have
 \begin{align}
  \log |f(z)|&\leq P^2\Big(4\E \log(P) + C\Big),\\
   \log |g(z)|&\leq P(P+1)\Big( 4\E \log(P) +  C \Big),
 \end{align}
 where $C = \log\left(\max(4\pi, 5d) 2^{4\E d} (4d+\|V\|_{\infty}+1)^{4e}\right)$.
\end{lemma}

\begin{proof}
 This is a computation.
\end{proof}

\begin{proof}[Proof of Theorem~\ref{thm:simple2}]
 The claim follows by Theorem~\ref{thm:caresti}.
\end{proof}

%%%%%%%%%%%%%%%%%%%%%%%%%%%%%%%%%%%%%%%%%%%%%%%%%%%%%%%%%%%%%%%%%%%%%%%%%%%%%%%%
%
%
%

\section{The absolutely continuous spectrum of a periodic operator}

The goal of this section is to prepare for the proof
of Theorem~\ref{thm:ac} given in the next section. The main
reason for writing a separate section, is to make this
section somewhat more expository.

Let $H$ be a $p$-periodic operator. For simplicity,
we will restrict ourself to considering $H$ in Fourier space,
i.e. $\widehat{H}:L^2(\T^d)\to L^2(\T^d)$
\be
 \widehat{H} f(x) = \left(\sum_{j=1}^{d} 2\cos(2\pi x_j)\right)f(x)
 +\sum_{k\in\mathbb{B}} \widehat{V}(k) f(x+k).
\ee
Given $f\in L^2(\T^d)$ and $y\in\mathbb{V}=[0,(p_1)^{-1})\times
\dots\times [0,(p_d)^{-1})$, we define
$f_y\in\ell^2(\mathbb{B})$ by $f_y(k) = f(k + y)$.
We have that $\widehat{H}_{x} f_{x} = (\widehat{H} f)_x$.
We recall that, we denote by $\psi(x,\ell)$ the orthonormal
basis of $\ell^2(\mathbb{B})$ consisting of eigenfunctions
of $\widehat{H}_{x}$. Thus, we have that
\be
 f_y = \sum_{\ell=1}^{P} \spr{\psi(y,\ell)}{f_y} \psi(y,\ell).
\ee
Hence, given a set $A \subseteq\mathbb{P}=\mathbb{V}\times
\{1,\dots,P\}$, it makes sense to define the projection
operator
\be\label{eq:defQA}
 (Q_A f)(y + k)=\sum_{\ell=1}^{P} \chi_{A}(y,\ell)
  \spr{\psi(y,\ell)}{f_y} \psi(y,\ell)_k,
\ee
where $y+k$ is the unique decomposition of $x\in\T^d$
into $y\in\mathbb{V}$ and $k\in\mathbb{B}$.
Note $I - Q_{A} = Q_{\mathbb{P}\setminus A}$
and that $Q_A$ is a projection.

\begin{proposition}\label{prop:bddQA}
 \begin{enumerate}
  \item Let $f:\T^d\to\C$ and $A\subseteq\mathbb{P}$. Then
   \be
    \|Q_A f\| \leq P^{\frac{3}{2}} |A|^{\frac{1}{2}} \|f\|_{L^\infty(\T^d)}.
   \ee
  \item If $A \subseteq B \subseteq \mathbb{P}$ then $Q_{A} \leq Q_{B}$.
 \end{enumerate}
\end{proposition}

\begin{proof}
 Let $A_1 = \{x:\quad\exists \ell:\ (x,\ell)\in A\}$ then
 $|A_1| \leq P |A|$. We compute
 \[
  \|Q_A f\|^2 = \int_{A_1} \sum_{\ell=1, (x,\ell)\in A}^{P} |\spr{\psi(x,\ell)}{f_x}|^2 dx
   \leq \int_{A_1} \sum_{\ell=1}^{P} \|f_x\|^2 dx
 \]
 As $\|f_x\|_{\ell^2(\mathbb{B})} \leq \sqrt{P} \|f\|_{L^{\infty}(\T^d)}$,
 (i) follows.

 To see that (ii) holds, observe that $Q_{B} - Q_{A} = Q_{B\setminus A}$.
 As $Q_{B\setminus A} \geq 0$, the claim follows.
\end{proof}

Next, we have that

\begin{lemma}\label{lem:specmeasperiodic}
 Let $\varphi\in\ell^1(\Z^d)$, $G\subseteq\mathbb{P}$,
 $\varphi_G = Q_{G}\varphi$, and define a measure $\mu_{G}$ by 
 \be
  \mu_{G}(A) = \spr{\varphi_{G}}{\chi_{A}(H) \varphi_{G}}.
 \ee
 Assume that 
 \begin{enumerate}
  \item For $(x,\ell) \in G$, we have that
   \be
    |\partial_{x_d} E(x,\ell)| \geq \gamma.
   \ee
  \item For $(x,\ell)\in G$, we have that the
   spectrum of $\widehat{H}_x$ is $\delta$-simple for some $\delta > 0$.
  \item For $(x,\ell)\in G$, we have
   \be
    \|\psi(x,\ell)\|_{\ell^1(\mathbb{B})} \leq C_1.
   \ee
 \end{enumerate}
 Then the measure $\mu$ is absolutely continuous and
 \be
  \left\|\frac{d\mu}{dE}\right\| \leq \frac{4 (C_1 \cdot 
   \|\varphi\|_{\ell^1(\Z^d)})^2}{\gamma}.
 \ee
\end{lemma}

\begin{proof}
 We have that
 \[
  \mu([E-\eps, E+\eps]) = \int_{\mathbb{V}} \sum_{\ell=1}^{P}
   \chi_{[E-\eps,E+\eps]}(E(x,\ell)) \chi_{G}(x,\ell)
    \cdot |\spr{\psi(x,\ell)}{\hat \varphi_x}|^2 dx.
 \]
 We first observe that
 \[
 |\spr{\psi(x,\ell)}{\hat \varphi_x}|\leq 
   \|\psi(x,\ell)\|_{\ell^1(\mathbb{B})} \cdot
    \|\hat{\varphi}_x\|_{\ell^{\infty}(\mathbb{B})}
     \leq C_1 \|\varphi\|_{\ell^1(\Z^d)}.
 \]
 Let $x'\in [0,(p_1)^{-1})\times\dots\times [0,(p_{d-1})^{-1})$.
 It thus suffices to bound
 \[
  I(\eps) =  \int_{0}^{(p_d)^{-1}} \sum_{\ell=1}^{P}
   \chi_{[E-\eps,E+\eps]}(E((x',x_d),\ell)) \chi_{G}(x', x_d,\ell) d x_d.
 \]
 By relabeling the eigenvalues, we may assume that they
 are analytic on small neighborhoods. Fix some $\ell$ and denote by
 $I$ the set of $x_d$ so that $(x', x_d, \ell) \in G$.
 Then if $[a,b]$ is a subinterval of $I$, we have by (iii)
 that
 \[
 |\{x_d \in [a,b]:\quad E((x',x_d),\ell) \in [E-\eps,E+\eps]\}|
   \leq \frac{2\eps}{\gamma}.  
 \]
 Due to the simplicity of eigenvalues, we have
 that (i) is stable. In particular if $\eps>0$ is small enough,
 $E((x',x_d), \ell) \in [E-\eps,E+\eps]$
 implies that there exists $|\ti{x}_d-x_d| \leq \frac{2\eps}{\gamma}$
 so that $E((x',\ti{x}_d), \ell) = E$. Hence, we are always
 in the case described above. Thus, we obtain
 \[
  I(\eps) \leq \#\{x_d,\ell :\quad E(x',x_d, \ell) = E\} \cdot \frac{4\eps}{\gamma}.
 \]
 By Proposition~\ref{prop:boundxd} and
 $|[0,(p_1)^{-1})\times\dots\times [0,(p_{d-1})^{-1})|
 = (p_1 \cdots p_{d-1})^{-1}$, the claim follows.
\end{proof}

%%%%%%%%%%%%%%%%%%%%%%%%%%%%%%%%%%%%%%%%%%%%%%%%%%%%%%%%%%%%%%%%%%%%%%%%%%%%%%%%
%
%
%

\section{Proof of absolutely continuous spectrum}

The goal of this section is to provide the proof
of Theorem~\ref{thm:ac}. It clearly suffices to prove
that the limit-periodic potentials obeying the
conditions given in the proof of Theorem~\ref{thm:genev}
have purely absolutely continuous
spectrum. One difference is that the conclusions of
Theorem~\ref{thm:simple} are not enough, but we will
need the full conclusions of Theorem~\ref{thm:simple2}.

For the readers convenience and easy reference, we summarize
the conclusions.
\begin{enumerate}
 \item There exist sets $\mathcal{G}_j\subseteq \mathbb{P}_j$
  with $|\mathcal{G}_j| \leq \eta_j = \frac{1}{(P_j)^2} \cdot \frac{1}{2^j}$.
 \item For $(x,\ell)\in\mathcal{G}_j$, we have
  \be
   |\partial_{x_d} E^{j}(x,\ell)| \geq \gamma_j
  \ee
  with $\gamma_j \geq 100 \delta_{j+1}^2$.
 \item For $(x,\ell)\in \mathcal{G}_j$ and $k \geq j$, there 
  is an unique $(x_k, \ell_k)$ such that 
  $(x,\ell) = A_j \cdots A_{k-1}(x_k,\ell_k)$.
 \item We have for some $|c|=1$ and $k\geq j$ that
  \be
   \|\psi^j(x,\ell) - c \psi^k(x_k,\ell_k)\|_{\ell^1} \leq 2 (\delta_j)^8.
  \ee
\end{enumerate}
We note that our choice of $\eta_j$ is different. Also
we need to choose $\eps_{j+1}$ such that $\gamma_j \geq 100\eps_{j+1} / \delta_{j+1}$,
which is not a problem.

Fix some $k \geq 1$ and for $j \geq k$ consider the
projections $P_{k,j} = Q_{G_{k,j}}$ as in \eqref{eq:defQA}
where
\be
 G_{k,j} = A_{k-1}^{-1} \cdots A_{j}^{-1} \mathcal{G}_j.
\ee

\begin{proposition}
 \begin{enumerate}
  \item $\|I - P_{j,j}\| \leq 2\eta_j P_j^2$.
  \item $\|P_{k+1, j} - P_{k,j}\| \leq \delta_k$.
   In particular, the limit $P_{\infty,j}=\lim_{k\to\infty}
   P_{k,j}$ exists.
  \item $\|I - P_{\infty,j}\|\leq 3\eta_j P_j^2$.
  \item $P_{\infty, j} \leq P_{\infty, j+1}$.
 \end{enumerate}
\end{proposition}

\begin{proof}
 (i) follows from Proposition~\ref{prop:bddQA}.  For (ii) observe,
 that
 \[
  P_{k, j} f (x+k) = \sum_{\ell} \chi_{A}(x,\ell)
   \spr{\psi^{k}(A_k(x,\ell))}{f} \psi^{k}(A_k(x,\ell))
 \]
 for $A = A_{k}^{-1} \cdots A_{j}^{-1} \mathcal{G}_j$.
 As $d(\psi^{k+1}(x,\ell), \psi^{k}(A_k(x,\ell))) \leq 2 \frac{\eps_{k+1}}{\delta_{k+1}}$,
 the bound on $\|P_{k+1,j} - P_{k,j}\|$ follows.
 To see convergence, observe that $\sum_{\ell\geq k} \delta_{\ell} \leq 2\delta_k$.
 This bound also implies (iii). Finally for (iv), we have that
 $P_{k,j+1} \geq P_{k,j}$. Thus this inequality also
 holds in the limit $k\to\infty$.
\end{proof}

\begin{proposition}\label{prop:specmeas}
 There exists $C_j > 0$ such that for $\varphi\in\ell^2(\Z^d)$
 with $\|\hat\varphi\|_{L^{\infty}(\T^d)}\leq 1$, we have
 for $k\geq j$
 \be
  \spr{P_{k,j} \varphi}{\chi_{[E-\eps,E+\eps]}(H^k) P_{k,j}
   \varphi} \leq C_j \eps.
 \ee
\end{proposition}

By property (ii), we have that
\be
 \|\psi^{k}(x_k, \ell_k)\|_{ \ell^{1}} \leq \sqrt{P^j} + 2 (\delta_j)^{8}.
\ee
For the proof, we need

\begin{lemma}
 Let $k\geq j$ and $(x,\ell)\in A_{k-1}^{-1} \cdots A_{j}^{-1} \mathcal{G}_{j}$.
 Then
 \be
  |\partial_{x_d} E^k(x,\ell)| \geq \frac{1}{2} \gamma_j.
 \ee
\end{lemma}

\begin{proof}
 Let $(\ti{x},\ti{\ell}) = A_{j} \cdots A_{k-1} (x,\ell)$. Then
 \[
  d(\psi^{k}(x,\ell), \psi^{j}(\ti{x},\ti{\ell}))\leq 3 \frac{\eps_{j+1}}{\delta_{j+1}}
   \leq (\delta_{j+1})^8.
 \]
 Next, observe that
 \[
  \partial_{x_d} E^j(\ti{x},\ti{\ell}) = \spr{\psi^{j}(\ti{x},\ti{\ell})}{
   \partial_{x_d} \widehat{H}_{\ti{x}}^j \psi^{j}(\ti{x},\ti{\ell})} = 
    \spr{\psi^{j}(\ti{x},\ti{\ell})}{
     \partial_{x_d} \widehat{H}_{\ti{x}}^k \psi^{j}(\ti{x},\ti{\ell})}
 \]
 and $|\partial_{x_d} E^j(\ti{x},\ti{\ell})|\geq\gamma_j \geq 100 (\delta_{j+1})^8$.
 Thus, the claim follows.
\end{proof}

\begin{proof}[Proof of Proposition~\ref{prop:specmeas}]
 This follows from Lemma~\ref{lem:specmeasperiodic}.
\end{proof}

We define now vectors $\varphi_{k,j} = P_{k,j} \varphi$ 
for $k\geq j$ and measures
\be
 \mu_{k,j}(A) = \spr{\varphi_{k,j}}{\chi_{A}(H^k) \varphi_{k,j}}.
\ee
We have that as $k\to\infty$, the vectors $\varphi_{k,j}$
converge to a limit $\varphi_{j}$ and we also
define the measure
\be
 \mu_{j}(A) = \spr{\varphi_{j}}{\chi_{A}(H) \varphi_{j}}.
\ee
As $H^k \to H$ and $\varphi_{k,j} \to \varphi_{j}$, we have that
$\mu_{k,j} \to \mu_{j}$ and in particular that $\mu_{j}$
is also absolutely continuous. Our results also imply
that $\mu_{j}(A) \geq \mu_{j-1}(A)$.

Define now a measure 
\be
 \mu(A) = \spr{\varphi}{\chi(H)\varphi}
\ee
As $\varphi_{j}\to\varphi$, we have that $\mu_j \to \mu$.

\begin{proof}[Proof of Theorem~\ref{thm:ac}]
 We may write
 \[
  \mu = \mu_1 + \sum_{j\geq 2} (\mu_j - \mu_{j-1}).
 \]
 As the measures $\mu_1$, $\mu_2-\mu_1$, $\mu_3-\mu_2$,
 \dots are all absolutely continuous and positive,
 it follows that $\mu$ is absolutely continuous.
 As we could choose $\varphi$ from a dense set,
 the claim follows.
\end{proof}

\appendix

%%%%%%%%%%%%%%%%%%%%%%%%%%%%%%%%%%%%%%%%%%%%%%%%%%%%%%%%%%%%%%%%%%%%%%%%%%%%5
%
%
%

\section{Cartan's estimate}

In this section, we will prove

\begin{theorem}\label{thm:caresti}
 Let $f:\C^d\to\C$ be an analytic function. Assume that there
 exists $y\in\C^d$ with $|y| > 1$
 such that $|f(y)|\geq \kappa$ and that
 we have
 \be
  \log\sup_{|z|\leq 4\E |y|} |f(z)| \leq A.
 \ee
 Then
 \be
 |\{ x\in [0,1]^d: \quad |f(z)| \leq \kappa \cdot 
   \left(\frac{\eps}{60\E^{3} d|y|}\right)^{d \cdot A} \}| \leq \eps.
 \ee
\end{theorem}

In order to prove this estimate, we will need the original
Cartan estimate.

\begin{theorem}
 Let $g:\C\to\C$ be an analytic function satisfying
 \be
  |g(y)| \geq \kappa
 \ee
 for some $y\in\C$ with $|y| > 1$. Then
 \be
  |\{x\in[0,1]:\quad |g(x)|\leq \delta \cdot \kappa\}| \leq \eps
 \ee
 for
 \be
  \log(\delta) = \log(\frac{\eps}{60\E^3 |y|}) \cdot \log(\sup_{|z|\leq 4 \E |y|} |g(z)|).
 \ee
\end{theorem}

\begin{proof}
 This is one version of Cartan's Estimate, see
 Theorem~11.3.4. in Levin's book \cite{lev}.
\end{proof}

\begin{proof}[Proof of Theorem~\ref{thm:caresti}]
 Define the function
 \[
  g_1(z) = f(z , y_2,\dots,y_{d}).
 \]
 Then $|g_1(y_1)| \geq \kappa$ and 
 $\log(\sup_{|z|\leq 4 \E |y_1|} |g_1(z)|)\leq A$.
 Hence, there exists a set $X_1\subseteq [0,1]$ of measure
 $\leq\frac{\eps}{d}$ such that for $x_1\in [0,1]\setminus X_1$,
 we have
 \[
  |f(x_1, y_2, \dots, y_{d})| = |g_1(x_1)| \geq 
   \kappa_1 = \kappa \cdot \left(\frac{\eps}{60\E^{3} d|y|}\right)^{A}.
 \]
 Applying this construction inductively to
 \[
  g_j(z) = f(x_1, \dots, x_{j-1}, z, y_{j+1}, \dots, y_d)
 \]
 with $x_\ell \in [0,1]\setminus X_\ell$,
 we obtain sets $X_1, \dots, X_d$ such that for
 $x_\ell \in [0,1]\setminus X_\ell$ for $1\leq\ell\leq j$,
 we have
 \[
  |f(x_1, \dots, x_{j}, y_{j+1}, \dots, y_d )| \geq 
   \kappa_j = \kappa \cdot \left(\frac{\eps}{60\E^{3} d|y|}\right)^{j \cdot A}.
 \]
 As
 \[
  |\{x\in [0,1]:\quad x_j \notin X_j\}| \geq 1 - |X_1| - \dots - |X_d| \geq 1- \eps
 \]
 the claim follows
\end{proof}

%%%%%%%%%%%%%%%%%%%%%%%%%%%%%%%%%%%%%%%%%%%%%%%%%%%%%%%%%%%%%%%%%%%%%5
%
%
%

\section{Distances of normalized eigenfunctions}

Let $X$ be a Hilbert space, and $\varphi,\psi$ two unit
vectors. We define the distance
\be
 d(\varphi,\psi) =\inf_{|c|=1} \|\varphi - c\psi\|.
\ee

\begin{theorem}\label{thm:distevs}
 Let $A$ be a self-adjoint operator on $X$ with
 \be
  \tr(P_{[-\delta,\delta]}(A)) = 1
 \ee
 and $A\psi = 0$, $\|\psi\|=1$. Assume the $\varphi$
 with $\|\varphi\| = 1$ satisfies $\|A\varphi\|\leq \eps$.
 Then
 \be
  d(\varphi,\psi) \leq \frac{2\eps}{\delta}.
 \ee
\end{theorem}

\begin{proof}
 Let $\varphi_1 = \spr{\psi}{\varphi}\psi$, $\varphi_2 = \varphi - \varphi_2$.
 Then $\eps\geq\|A\varphi\| = \|A\varphi_2\|\geq\delta\|\varphi_2\|$.
 Thus $|\spr{\psi}{\varphi}| = \|\varphi_1\| \geq 1 - \frac{\eps}{\delta}$.
 Taking $c = \spr{\psi}{\varphi}/|\spr{\psi}{\varphi}|$
 the claim follows.
\end{proof}

More sophisticated versions of this argument can be found
in Section~9 of \cite{kskew}. In particular, the methods
discussed there would allow one to understand the set
$\{(x,\ell)\in\mathbb{V}:\quad E(x,\ell)=E\}$ for any
$E\in\R$.

%%%%%%%%%%%%%%%%%%%%%%%%%%%%%%%%%%%%%%%%%%%%%%%%%%%%%%%%%%%%%%%%%%%%%%%
%
%
%

\end{document}